\documentclass{article}
\usepackage{graphicx}
\usepackage{amsmath,amsfonts,amssymb,mathtools,amsthm}
\usepackage{url}
\usepackage{comment}
\usepackage{color}
\usepackage{subcaption}
\usepackage{algorithm}
\usepackage[noEnd=true]{algpseudocodex}
\usepackage{float}
\usepackage[shortlabels]{enumitem}
\usepackage{multirow}
\usepackage{diagbox}
\usepackage{hyperref}
\newtheorem{theorem}{Theorem}[section]
\newtheorem{proposition}[theorem]{Proposition}
\newtheorem{lemma}[theorem]{Lemma}

\newtheorem{claim}{Claim}
\theoremstyle{definition}
\newtheorem{Definition}[theorem]{Definition}
\newtheorem{example}[theorem]{Example}

\newtheorem{remark}[theorem]{Remark}

\usepackage[affil-it]{authblk}

\begin{document}
	
	\title{Stability Analysis of Generalized Multi-Source Weber Problems}
	
	\author[1,2]{N.H. Hai\thanks{Email: hoanghai10ly@gmail.com}}
	\author[1,2]{V.S.T. Long\thanks{Email: vstlong@hcmus.edu.vn}}
	\author[3]{N.M. Nam\thanks{Email: mnn3@pdx.edu}}
	\author[4]{N.D. Yen\thanks{Email:
			ndyen@math.ac.vn}}
	
	\affil[1]{Faculty of Mathematics and Computer Science, University of Science, Ho Chi Minh City, Vietnam}
	\affil[2]{Vietnam National University, Ho Chi Minh City, Vietnam}
	\affil[3]{Fariborz Maseeh Department of Mathematics and Statistics, Portland State University, Portland, OR 97207, USA} 
	\affil[4]{Institute of Mathematics, Vietnam Academy of Science and Technology, 18 Hoang Quoc Viet, Hanoi 10307, Vietnam}
	%\date{} % Bỏ ngày nếu không cần thiết
	\maketitle

	\begin{abstract}
		\noindent This paper presents a stability analysis of the generalized multi-source Weber problems under the influence of data perturbation within the framework of the Minkowski function. First, we establish explicit Lipschitz constants for the objective functions and optimal value functions. Then, we prove several properties of the global optimal solution sets. Furthermore, we provide sufficient conditions for the upper semicontinuity of global solution mappings and the inner semicontinuity of local solution mappings. Three illustrative examples are constructed.
		Our results give indirect answers to several open questions regarding the stability of local optimal solution mappings of the optimization problems discussed.		
	\end{abstract}
	
	

	\section{Introduction}
	
	The origins of location science can be traced back to the seventeenth century, beginning with Fermat's problem of minimizing the sum of distances to three given points in the plane, later solved by Torricelli and now known as the Fermat--Torricelli problem. This fundamental model--minimizing aggregate distance --has evolved from a purely geometric construction into a central problem of facility location theory, with applications ranging from logistics and transportation planning to data analysis and clustering; see, for example,~\cite{Aggarwal2014,AragonArtacho2022,Bagirov2011,Bagirov2006,  Brusco2006, cuong2023global,CTWYOptim,cuong2020qualitative,Jain2010, Kantardzic2011,Kuhn1973,Martini2002,Mordukhovich2010,Nam2017,Nam2018, Ordin2015} and the references therein.
	
	In recent years, considerable attention has been devoted to the multi-source Weber problem, which constitutes a significant extension of the classical Fermat--Torricelli model. The problem has been investigated by Cuong,~T.H., Yao,~J.-C., Yen, N.D., and their coauthors, from both theoretical and numerical aspects. In particular, several sufficient conditions ensuring the existence of global and local solutions were established in~\cite{cuong2023global,CTWYOptim,cuong2020qualitative}. Beyond existence results, various stability properties of the optimal value function and of both global and local solution mappings were analyzed in~\cite{Cuong2024,cuong2020qualitative}. 
	
	It is worth noting that most classical location models rely on the Euclidean norm. However, practical applications often require more flexible distance measures in order to capture different distance preferences, unequal transportation costs, or varying operating conditions. The observation has recently motivated a shift toward the framework of Minkowski function, which provides a robust and unified generalization encompassing a wide class of norms. Within that new setting, \textit{generalized multi-source Weber problems} (GMWPs) have been studied with respect to solution existence and numerical algorithms; see, for instance,~\cite{LNTV}. However, the sensitivity of optimal solutions to data perturbation—a crucial question since real-world data often contains uncertainty—has not been studied until now. This paper fills that gap by providing a detailed stability analysis of GMWPs.
	
	The paper is organized as follows. Section~\ref{sec2} presents some concepts and results of convex analysis and set-valued analysis, along with a formulation of generalized multi-source Weber problems. In Section~\ref{sec3}, we prove the Lipschitz continuity of the objective functions and the optimal value functions of the problems. Section~\ref{sec4} is devoted to the study of the global optimal solution sets. Specifically, we first show that the solution sets lie within certain compact sets, and then provide sufficient conditions for the upper semicontinuity (or lack of upper semicontinuity) of the global solution mappings. In Section~\ref{sec5}, we obtain conditions guaranteeing the inner semicontinuity of local solution mappings. The paper concludes with final remarks in Section~\ref{sec6}.
	
	Throughout the paper, the set of positive integers and the set of real numbers are denoted by $\mathbb{N}$ and $\mathbb{R}$, respectively. For every $d \in \mathbb{N}$, the Euclidean norm on $\mathbb{R}^d$ is denoted by $\|\cdot\|$. 
	The closed ball and the open Euclidean ball centered at $c \in \mathbb{R}^d$ with radius $r > 0$ are denoted by $\mathbb{B}[c; r]$ and $\mathbb{B}(c; r)$, respectively. 
	For a set $C \subset \mathbb{R}^d$, we use $\operatorname{int}(C)$, $\operatorname{cl}(C)$, and $\operatorname{bd}(C)$ to denote, respectively, the interior, the closure, and the boundary of~$C$.

	\section{Auxiliary results and problem settings}
	\label{sec2}
	
	In what follows, let $\mathbb{R}^{n}$, $\mathbb{R}^{nm}$, and $\mathbb{R}^{nk}$, where $n,m,k\in \mathbb{N}$, be equipped with the Euclidean norm.
	
	\subsection{Auxiliary results}
	The following concepts and results are important for our further studies. For more details, the interested reader can refer to \cite{Bauschke,boyd,mordukhovich2023easy,Hoangtuy}.
	
	\begin{Definition}\label{Minkowski_f}
		Let $F$ be a nonempty compact convex subset of $\mathbb R^n$ such that ~$0\in \mbox{\rm int}(F)$. The {\em Minkowski function} associated with $F$  is defined by
		\begin{equation}
			\label{gauge}
			\rho_F(x) = \inf\big\{t\geq 0\; |\; x\in tF\big\}, \; x\in \mathbb{R}^n.
		\end{equation}
		The {\em polar set} $F^\circ$ of $F$ is the set
		$$F^{\circ} = \big\{v\in \mathbb{R}^n\mid \langle v, x\rangle \leq 1,\; \forall x\in F\big\}.$$
	\end{Definition}
	
	We define
	$$\|F\| = \sup\{\|x\|\mid x\in F\}\; \text{ and }\;\|F^\circ \| = \sup\{\|v\|\mid  v\in F^\circ\}.$$
	
	Several fundamental properties of the Minkowski function are recalled in the next lemma. 
	
	\begin{lemma}\label{lemmarho0} {\rm (See \cite[Proposition 2.1]{colombo2004subgradient}, \cite[Lemma~4.1]{longoptimletter}, and \cite[Theorem~6.14 and Proposition~6.18]{mordukhovich2023easy})}
		Let $\rho_F$ be the Minkowski function defined in \eqref{gauge}. Then the following properties hold:
		\begin{enumerate}
			\item [{\rm (i)}] $\rho_F(\alpha x) =\alpha\rho_F(x)$ for all $\alpha \geq0$,  $x\in \mathbb{R}^n$.
			\item [{\rm (ii)}] $\rho_F(x_1+x_2)\leq \rho_F(x_1)+\rho_F(x_2)$ for all $x_1,x_2\in \mathbb{R}^n$.
			\item [{\rm (iii)}] $\rho_F(x)=0\;\text{ if and only if }\; x=0.$
			\item [{\rm (iv)}]  $\rho_F$ is $\|F^\circ\|$-Lipschitz continuous on $\mathbb{R}^n$.
			\item [{\rm (v)}] $x\in \operatorname{bd}(F)$ if and only if $\rho_F(x)=1$.  
			\item [{\rm (vi)}] $\rho_{F^\circ}(v)=\max\limits_{u\in F}\langle v,u\rangle$ for all $v\in \mathbb{R}^n.$
			\item [{\rm (vii)}] $\dfrac{\rho_F(x)}{\|F^\circ\|}\leq \|x\|\leq \|F\|\,\rho_F(x)$ for all $x\in\mathbb{R}^n.$
			\item [{\rm (viii)}] $\rho_F (x)\leq \|F\|\|F^{\circ}\|\rho_F(-x)$ for all $x\in \mathbb{R}^n$.
		\end{enumerate}
	\end{lemma}
	
	Given any~$c \in \mathbb{R}^n$ and~$r \geq 0$, define the generalized closed ball by
	\begin{equation}\label{geneball}
		\mathbb{B}_F[c; r] = \big\{x \in \mathbb{R}^n \mid \rho_F(x - c) \leq r\big\}.
	\end{equation}

	\begin{lemma}\label{lm2} The generalized closed ball $\mathbb B_F[c;r]$ is compact.
	\end{lemma}
	\begin{proof} For any $x\in \mathbb B_F[c;r]$, by Lemma~\ref{lemmarho0}(vii), we have
		\begin{equation*}
			\|x-c\|\leq \|F\|\rho_F(x-c)\leq r\|F\|.
		\end{equation*}
		 It follows that $\mathbb B_F[c; r]\subset \mathbb B[c; r\|F\|]$; hence $\mathbb B_F[c;r]$ is bounded. The closedness of $\mathbb B_F[c; r]$ follows from~\eqref{geneball} and the continuity of the Minkowski function $\rho_F$ (see Lemma \ref{lemmarho0}(iv)). Therefore, $\mathbb B_F[c; r]$ is compact.
	\end{proof}
	\begin{remark}\label{remarkcompact}
		 By Tychonoff's Theorem (see, e.g., \cite[Chapter~5, Theorem~13]{kelly}),  if $Q\subset \mathbb{R}^n$ is compact, then the set  $Q^k = Q \times \dots \times Q$ (formed by $k$ copies of $Q$) is also compact in the product space $\mathbb{R}^{nk}$.
	\end{remark}
	
	We now recall several concepts of set-valued analysis, which can be  found in the books by Rockafellar and Wets~\cite{r} and Mordukhovich~\cite{mordukhovich2018variational}.
	
	The \textit{domain} and \textit{graph} of 
	a set-valued mapping $G: \mathbb{R}^{m} \rightrightarrows \mathbb{R}^{k}$ are defined, respectively, by 
$\operatorname{dom}(G) = \big\{a \in 	\mathbb{R}^{m} \mid G(a) \neq \emptyset\big\}$ and $$\operatorname{gph} (G) = \big\{(a, x) \in 	\mathbb{R}^{m} \times \mathbb{R}^{k} \mid x \in G(a)\big\}.$$
	
	\begin{Definition}\label{defusclsc}
		Let  $G: \mathbb{R}^{m} \rightrightarrows \mathbb{R}^{k}$ be a set-valued mapping and $\bar{a} \in \operatorname{dom} (G)$. We say that
		\begin{enumerate}[(i)]
			\item $G$ is \textit{upper semicontinuous} at $\bar{a}$ if for any open set $V \subset \mathbb{R}^{k}$ containing $G(\bar{a})$, there exists a neighborhood $U$ of $\bar{a}$ such that $G(a) \subset V$ for all $a \in U$;
			\item $G$ is \textit{lower semicontinuous} at $\bar{a}$ if for any open set $V \subset \mathbb{R}^{k}$ such that $G(\bar{a}) \cap V \neq \emptyset$, there exists a neighborhood $U$ of $\bar{a}$ such that $G(a) \cap V \neq \emptyset$ for all $a \in U$;
			\item $G$ is \textit{continuous} at $\bar{a}$ if it is both upper and lower semicontinuous at this point;
			\item $G$ is \textit{inner semicontinuous} at $(\bar{a},\bar{x}) \in \operatorname{gph}(G)$ if for every open set $V\subset \mathbb{R}^{k}$ with $\bar{x}\in V$, there is a neighborhood $U$ of  $\bar a$ such that $V\cap G(a)\neq \emptyset$ for all $a\in U$;
			\item $G$ is \textit{closed} at $\bar{a}$ if for any sequence $(a_\ell, x_\ell) \in \operatorname{gph} (G)$ with $(a_\ell, x_\ell) \to (\bar{a}, \Bar{x})$, one has $\Bar{x} \in G(\bar{a})$.
		\end{enumerate}
	\end{Definition}

 Given a function $g: \mathbb{R}^{k} \times \mathbb{R}^{m} \to \mathbb{R}$, consider the parametric optimization problem 
	\begin{align}\label{(P)}
		\min_{x \in \mathbb{R}^{k}} g(x, a),
	\end{align} with $a\in \mathbb{R}^m$ being a parameter. The \textit{optimal value function} associated with~\eqref{(P)} is defined by 
	$v(a) = \inf\limits_{x \in \mathbb{R}^{k}} g(x, a)$ for $a \in \mathbb{R}^{m}.$ The \textit{global solution mapping} $S: \mathbb{R}^{m} \rightrightarrows \mathbb{R}^{k}$ of~\eqref{(P)} is defined by 
			$S(a) = \big\{ \Bar{x} \in \mathbb{R}^{k} \mid g(\Bar{x}, a) = v(a) \big\}.$
			
  The next lemma will be essential for our later analysis.
	
	\begin{lemma}\label{upper}  Consider problem~\eqref{(P)} and assume that the function $g$ and the optimal value function $v$ are continuous. For a point $\bar{a} \in \mathbb{R}^{m}$, if there are a neighborhood $U$ of $\bar{a}$ and a compact set $W \subset \mathbb{R}^{nk}$ such that $S(a) \subset W$ for all $a \in U$, then the global solution mapping $S$ is upper semicontinuous at $\bar{a}$.
	\end{lemma}
	
	\begin{proof} If $S$ is not upper semicontinuous at $\bar{a}$, then there is an open set $V \subset \mathbb{R}^{k}$ with $S(\bar{a}) \subset V$ such that for every neighborhood $U_\ell = \mathbb{B}(\bar{a}; 1/\ell) \cap U$, $\ell \in \mathbb{N}$, of~$\bar{a}$, there exists at least one point $a_\ell \in U_\ell$ and an element $x_\ell \in S(a_\ell)$ such that $x_\ell \notin V$.
		  Clearly, the sequence $(a_\ell)$ converges to $\bar{a}$ as $\ell \to \infty$. Furthermore, since $a_\ell \in U$ for all $\ell$, our assumptions guarantee that $$x_\ell \in S(a_\ell) \subset W.$$ 
		 By the compactness of $W$, there is a subsequence $(x_{\ell'})$ of the sequence $(x_{\ell})$ that converges to some point $\Bar{x} \in W$.
		We now show that $\Bar{x} \in S(\bar{a})$. Indeed,  since $x_{\ell'} \in S(a_{\ell'})$, one has 
		\[ g(x_{\ell'}, a_{\ell'}) = v(a_{\ell'}). \]
		 Passing this equality to the limit as $\ell' \to \infty$, and utilizing the continuity of the functions $g$ and $v$, we obtain
		$g(\Bar{x}, \bar{a}) = v(\bar{a}).$
		 Hence, $\Bar{x} \in S(\bar{a})$.
		However, by our  construction, $x_{\ell'} \notin V$ for all $\ell'$. Since  $\mathbb{R}^{k}\setminus V$ is a closed set, we have $ \Bar{x} \notin V.$ 
		This contradicts the fact that $$\Bar{x} \in S(\bar{a}) \subset V.$$ 
		 Thus, $S$ must be upper semicontinuous at $\bar a$.
	\end{proof}
	
	%%%%%%%%%%%%%%%%%%%%%%%%%
	%%%%%%%%%%%%%%%%%%%%%%%%%%%%\
	
	\subsection{Generalized multi-source Weber problems}
	\label{subsec3}
	
	Let $A = \{a^1, \dots, a^m\} \subset \mathbb{R}^{n}$ be a finite collection of data points (also called demand points) and let $X = \{x^1, \dots, x^k\} \subset \mathbb{R}^{n}$ represent $k$ facilities (or centers), which are to be found. Put $I = \{1, \dots, m\}$ and $J = \{1, \dots, k\}.$ 
	
	In what follows, it is convenient for us to identify the above sets $A$ and $X$ with the vectors $a = (a^1, \dots, a^m) \in \mathbb{R}^{nm}$ and $x = (x^1, \dots, x^k) \in \mathbb{R}^{nk}$, respectively. Assume that $1 \leq k \leq m.$
	
	Let $F\subset\mathbb R^n$ be a compact convex set with~$0\in \mbox{\rm int}(F)$ and let~$\rho_F$be the Minkowski function defined by~\eqref{gauge}. 
	
	The \textit{generalized multi-source Weber problem} (GMWP) is the minimization problem
		\begin{equation}\label{maxminn} 
		\min_{x \in \mathbb{R}^{nk}} \left\{ f(x, a) = \sum_{i\in I} \min_{j \in J} \rho_F(x^j - a^i) \right\}, 
	\end{equation} where the data tube $a = (a^1, \dots, a^m) \in \mathbb{R}^{nm}$ plays the role of a parameter.
	
	\begin{remark} 
	In the special case where $F$ is the Euclidean unit ball, problem~\eqref{maxminn} reduces to the multi-source Weber problem recently investigated in~\cite{cuong2023global,Cuong2024,CTWYOptim}.
	\end{remark}
	
	\begin{remark} Arguing similarly as in ~\cite[Section~3]{CTWYOptim}, we can transform~\eqref{maxminn} to a DC programming problem, where DC stands for ``difference of convex functions''. Consequently,~\eqref{maxminn} can be efficiently solved by the DC algorithm (DCA) introduced in \cite{TA1, Tao1986} or the boosted DC algorithm in~\cite{AragonArtacho2022}. 
	\end{remark}

	\begin{Definition}\label{deflocalsolution}
		A vector $\bar{x} \in \mathbb{R}^{nk}$ is said to be a \textit{local optimal solution} to  problem~\eqref{maxminn} if there exists $\varepsilon > 0$ such that 
		\begin{equation}\label{localsolution}
			 f(\bar{x}, a) \leq f(x, a)          
		\end{equation}
		for all $x \in \mathbb{B}(\bar{x};\varepsilon)$.
		 If the inequality~\eqref{localsolution} is \textit{strict} for all $x \in \mathbb{B}(\bar{x}; \varepsilon) \setminus \{\bar{x}\}$, then $\bar{x}$ is said to be a \textit{strict local optimal solution} to problem~\eqref{maxminn}.
	\end{Definition}

	Given $a \in \mathbb{R}^{nm}$, by $S(a)$ and $S_{loc}(a)$ we denote the sets of global and local optimal solutions to~\eqref{maxminn}, respectively. So, we have

	\begin{equation}\label{global_sols_W}
		S(a) = \left\{ \bar{x} \in \mathbb{R}^{nk} \mid f(\bar{x}, a) = \min_{x \in \mathbb{R}^{nk}} f(x, a) \right\}.
	\end{equation}
	
	%%%%%%%%%%%%%%%%%%%%%%%%%%%%%%%%%%%%%%%%%%%
	%%%%%%%%%%%%%%%%%%%%%%%%%%%%%%%%%%%%%%%%%%
	\section{Lipschitz stability for generalized multi-source Weber problems}
	\label{sec3}
	
In this section, we derive explicit Lipschitz constants for some functions associated with~\eqref{maxminn}. We first establish the Lipschitz continuity of the objective function.
	
	\begin{proposition}\label{proobjectivefuntion}
			The function $f$ defined in~\eqref{maxminn} is Lipschitz continuous on $\mathbb{R}^{nk}\times \mathbb{R}^{nm}$. More precisely, $L = m\sqrt{2}\|F^\circ\|$ is a 
	 Lipschitz constant of $f$.
	\end{proposition}
	
	\begin{proof}
		For each pair $(i, j) \in I \times J$, consider the function $\varphi_{i,j} \colon \mathbb{R}^{nk} \times \mathbb{R}^{nm} \to \mathbb{R}$ defined by
		$\varphi_{i,j}(x, a) = \rho_F(x^j - a^i)$ for all $(x, a) \in \mathbb{R}^{nk} \times \mathbb{R}^{nm}$. The Lipschitz continuity of  $f$ is established through the following three claims, where the third one can be easily verified.
		
		\begin{claim}
			Each function $\varphi_{i,j}$ is Lipschitz continuous on $\mathbb{R}^{nk} \times \mathbb{R}^{nm}$ with the constant $L_0 = \sqrt{2}\|F^\circ\|$.
		\end{claim}
		
		Indeed, fix any $(\bar{x}, \bar{a})$ and $(x, a)$ in $\mathbb{R}^{nk} \times \mathbb{R}^{nm}$. Using the subadditivity together with the Lipschitz continuity of $\rho_F$ (see Lemma \ref{lemmarho0}(ii) and (iv)) and the Cauchy-Schwarz inequality, one has
		\begin{equation*}
			\begin{aligned}
				\varphi_{i,j}(\bar{x}, \bar{a}) - \varphi_{i,j}(x, a) 
				&= \rho_F(\bar{x}^j - \bar{a}^i) - \rho_F(x^j - a^i) \\
				&\le \rho_F\left((\bar{x}^j - x^j) + (a^i - \bar{a}^i)\right) \\
				&\le \rho_F(\bar{x}^j - x^j) + \rho_F(a^i - \bar{a}^i) \\
				&\le \|F^\circ\| \left( \|\bar{x}^j - x^j\| + \|a^i - \bar{a}^i\| \right) \\
				&\le \sqrt{2}\|F^\circ\| \sqrt{\|\bar{x}^j - x^j\|^2 + \|a^i - \bar{a}^i\|^2} \\
				&\le \sqrt{2}\|F^\circ\| \|(\bar{x}, \bar{a}) - (x, a)\|.
			\end{aligned}
		\end{equation*}
		By interchanging the roles of $(\bar{x}, \bar{a})$ and $(x, a)$, we obtain the same upper bound for $\varphi_{i,j}(x, a) - \varphi_{i,j}(\bar{x}, \bar{a})$.  Thus $\varphi_{i,j}$ is Lipschitz with constant $L_0 = \sqrt{2}\,\|F^\circ\|$.
		
		\begin{claim} 
			Let $\phi_1, \ldots, \phi_s \colon \mathbb{R}^r \to \mathbb{R}$ be Lipschitz continuous functions with positive constants $\ell_1, \ldots, \ell_s$, respectively. Define
			\[
			\phi_{\min}(z) = \min_{1\le p\le s} \phi(z)\quad\text{for } z\in \mathbb{R}^r.
			\]
			Then $\phi_{\min}$ is Lipschitz continuous on $\mathbb{R}^r$ with constant
			$\ell = \max\limits_{1\le p\le s}\ell_p$.
		\end{claim}
		Indeed, take any $z,\bar{z}\in \mathbb{R}^r$. Let $\hat{p} \in \{1, \dots, s\}$ be such that 
			 $\phi_{\hat{p}}( z) = \phi_{\min}( z)$.	It follows that
		\begin{equation*}
			\phi_{\min}(\bar{z}) - \phi_{\min}(z) \le \phi_{\hat{p}}(\bar{z}) - \phi_{\hat{p}}(z) \le \ell_{\hat{p}}\|\bar{z}-z\| \le \ell\|\bar{z}-z\|.
		\end{equation*} Since the inequality $\phi_{\min}(z) - \phi_{\min}(\bar{z}) \le \ell\|z-\bar{z}\|$ can be proved similarly, 
		the conclusion is valid.

		\begin{claim}
			A finite sum of Lipschitz continuous functions is Lipschitz continuous with a constant equal to the sum of their individual Lipschitz constants.
		\end{claim}
		
	 Now we can conclude the proof of the proposition. For each $i\in I$, define
	\begin{equation}\label{psi_i}
		\psi_i(x,a) = \min_{j\in J} \varphi_{i,j}(x,a).
	\end{equation}
	 From~\eqref{maxminn} and~\eqref{psi_i}, we obtain
	$f(x,a) = \sum\limits_{i\in I}\psi_i(x,a).$ 	By Claim~1, each function $\varphi_{i,j}$ has the Lipschitz constant
			$L_0=\sqrt{2}\,\|F^\circ\|$. Hence, Claim~2 and~\eqref{psi_i} imply that $\psi_i$ is Lipschitz continuous
			with the same constant $L_0$. 
  Then, by Claim~3 we see that $f$ is Lipschitz continuous with constant $L = m\sqrt{2}\,\|F^\circ\|$.
	\end{proof}

	\begin{Definition} \label{defoptimalvaluefunctions}
  Depending on the variable $a = (a^1, \dots, a^m) \in \mathbb{R}^{nm}$, the \textit{optimal value function $\mathcal{V}$ of~\eqref{maxminn}} is defined by
	\begin{equation} \label{valuefunctionwp}
				\mathcal{V}(a) = \min_{x \in \mathbb{R}^{nk}} f(x, a) = \min_{x \in \mathbb{R}^{nk}}\left\{ \sum_{i\in I}\min_{j \in J} \rho_F(x^j - a^i)\right\}.
			\end{equation}
	\end{Definition}
	
	The following theorem is the main result of this section.
	 
	\begin{theorem} \label{thm:vkW_lipschitz}
		The optimal value function $\mathcal{V}$ of~\eqref{maxminn} is $\sqrt{m}\|F^\circ\|$-Lipschitz continuous on $\mathbb{R}^{nm}$.
	\end{theorem}
	
	\begin{proof}
		Let $a = (a^1, \dots, a^m)$ and $\bar{a} = (\bar{a}^1, \dots, \bar{a}^m)$ be two arbitrary vectors in~$\mathbb{R}^{nm}$. For any $\varepsilon > 0$, it follows from \eqref{valuefunctionwp} that there exists  $x = (x^1, \dots, x^k) \in \mathbb{R}^{nk}$ such that
		\begin{equation*} \label{eq:eps_opt}
			f(x, a) \leq \mathcal{V}(a) + \varepsilon.
		\end{equation*}
		Furthermore, applying~\eqref{valuefunctionwp} for $a=\bar a$ yields  $\mathcal{V}(\bar{a}) \leq f(x, \bar{a})$. Consequently, 
		\begin{equation} \label{eq:value_diff}
			\mathcal{V}(\bar{a}) - \mathcal{V}(a) \leq f(x, \bar{a}) - f(x, a) + \varepsilon.
		\end{equation}
			For any $i \in I$ and $j \in J$, invoking the subadditivity and the Lipschitz continuity property of $\rho_F$ from Lemma \ref{lemmarho0}(ii) and (iv), we obtain
		\begin{align*}
			\rho_F(x^j - \bar{a}^i) &\leq \rho_F(x^j - a^i) + \rho_F(a^i - \bar{a}^i)\\ 
			& = \rho_F(x^j - a^i) + \big(\rho_F(a^i - \bar{a}^i)-\rho_F(0)\big)\\ 
			&\leq \rho_F(x^j - a^i) + \|F^\circ\| \|a^i - \bar{a}^i\|.
		\end{align*}
		Taking the minimum over $j \in J$ gives
			\begin{equation*}\label{estimate_i}
				\min_{j \in J} \rho_F(x^j - \bar{a}^i) \leq \min_{j \in J} \rho_F(x^j - a^i) + \|F^\circ\| \|a^i - \bar{a}^i\| 
			\end{equation*} for all $i \in I$.
		 Summing up these inequalities over $i \in I$, we get
		\begin{equation} \label{eq:gk_sum_diff}
			f(x, \bar{a}) \leq f(x, a) + \|F^\circ\| \sum_{i \in I} \|a^i - \bar{a}^i\|.
		\end{equation}
		Moreover, it follows from the Cauchy-Schwarz inequality that
		\begin{equation*} \label{eq:cauchy_bound}
			\sum_{i \in I} \|a^i - \bar{a}^i\| \leq \sqrt{m} \left( \sum_{i \in I} \|a^i - \bar{a}^i\|^2 \right)^{1/2} = \sqrt{m} \|a - \bar{a}\|.
		\end{equation*}
		Hence, by~\eqref{eq:value_diff} and~\eqref{eq:gk_sum_diff} we obtain
		\begin{equation*}
			\mathcal{V}(\bar{a}) - \mathcal{V}(a) \leq \sqrt{m} \|F^\circ\| \|a - \bar{a}\| + \varepsilon.
		\end{equation*}
		Since $\varepsilon > 0$ is arbitrary, it follows that $$\mathcal{V}(\bar{a}) - \mathcal{V}(a) \leq \sqrt{m} \|F^\circ\| \|a - \bar{a}\|.$$ By swapping the roles of $a$ and $\bar{a}$, we also have
		 $$\mathcal{V}(a) - \mathcal{V}(\bar{a}) \leq \sqrt{m} \|F^\circ\| \|a - \bar{a}\|.$$
		So, the inequality $	
			|\mathcal{V}(\bar{a}) - \mathcal{V}(a)| \leq \sqrt{m} \|F^\circ\| \|a - \bar{a}\|$ is valid for all $a, \bar a\in \mathbb{R}^{nm}$. This completes the proof. \end{proof}
				
	\section{Stability analysis of global solution mappings} \label{sec4}
	
To analyze the stability of the global solution set of~\eqref{maxminn} w.r.t. the changes of the vector $a = (a^1, \dots, a^m) \in \mathbb{R}^{nm}$, representing the data points of~\eqref{maxminn}, we need to put the procedure called natural clustering from~\cite{cuong2023global} (see~\cite{cuong2020qualitative} for the original version applied to the minimum sum-of-squares clustering problem) in our more general setting. 

Given a vector $x=(x^1,\dots,x^k)\in \mathbb{R}^{nk}$, we define a partition of the data set $A$ into $k$ subsets $\{A^1, \dots, A^k\}$ through the following inductive procedure. Starting with $A^0 = \emptyset$, each subset $A^j$ for $j\in J$ is defined by
	\begin{equation}\label{Aj}
		A^j = \left\{ a^i \in A \setminus \bigcup_{q=0}^{j-1} A^q \;\; \bigg| \;\; \rho_F(x^j - a^i) = \min_{s\in J} \rho_F(x^s - a^i) \right\}.
	\end{equation}
	The resulting collection $\{A^1, \dots, A^k\}$ is referred to as the \textit{natural clustering} associated with the centers system $x = (x^1, \dots, x^k) \in \mathbb{R}^{nk}$. Based on this procedure, we have the following definition.
	
	\begin{Definition}\label{attraction_1}
		A component $x^j$ of $x = (x^1, \dots, x^k) \in \mathbb{R}^{nk}$ is said to be \emph{attractive} with respect to $A$ if the set
		\begin{equation}\label{attraction}
			A[x^j] = \left\{ a^i \in A \mid \rho_F(x^j - a^i) = \min_{s\in J} \rho_F(x^s - a^i) \right\},
		\end{equation} called the \textit{attraction set} of $x^j$,
		is nonempty.
	\end{Definition}
	
	\noindent It follows directly from~\eqref{Aj} and~\eqref{attraction} that 
	$$A^j = A[x^j] \setminus \bigcup_{q=1}^{j-1} A^q\quad\; (j\in J).$$
	
	For every $x=(x^1,\dots,x^k)\in \mathbb{R}^{nk}$, where $k \ge 2$, define the vectors $$x^{-j} = \big(x^1,\dots,x^{j-1},x^{j+1},\dots,x^k\big)\in \mathbb{R}^{n(k-1)}\quad\; (j\in J).$$

		Although the existence of global solutions for GMWP was established in \cite{LNTV}, no bounds for the centers $x^j$, $j\in J$, were given there. We can identify specific closed balls whose intersection contains each component of a global solution as follows.

	\begin{theorem}\label{Compactball}
		Let $a=(a^1,\ldots,a^m) \in \mathbb{R}^{nm}$ be such that $a^i \neq a^p$ for all $i, p \in I$ with $i \neq p$. Then, for any  $x=(x^1,\ldots,x^k)$ from the global solution set $S(a)$ of~\eqref{maxminn}, one has
		\begin{equation}\label{coverball}
			x^j \in \bigcap_{i\in I}\mathbb{B}_F[a^{i}; r^{i}] \quad \text{for all } j \in J,
		\end{equation}
		where 	\begin{equation}\label{indexradius}
			r^{i} = (\|F\| \|F^\circ\| + 1) \max_{p \in I} \rho_F(a^{i} - a^p)
		\end{equation}
		and the closed ball $\mathbb{B}_F[a^{i}; r^{i}]$, with $i\in I$, is defined as in \eqref{geneball}. In addition,~$S(a)$ is compact.
	\end{theorem}
	
	\begin{proof} Observe first that the set $S(a)$ is nonempty and closed by Theorem~3.2 in \cite{LNTV}. To prove~\eqref{coverball} by contradiction, suppose that there exists an element
		\[
		y=(y^1,\dots,y^k)\in S(a)
		\]
		such that one of its components does not belong to $\bigcap\limits_{i\in I}\mathbb B_F[a^i;r^i]$. By reordering the system of facilities $(y^1,\dots,y^k)$, if necessary, we may assume that the component is~$y^k$. Then, we can find some~$i_0\in I$ such that
		\begin{equation}\label{notin}
			y^k\notin \mathbb{B}_F[a^{i_0}; r^{i_0}].
		\end{equation}
	For every $i\in I$, we have
		{\small\begin{equation*} 
				\begin{aligned}
					\rho_F( y^k - a^i) &\ge \rho_F(y^k - a^{i_0}) - \rho_F(a^i - a^{i_0})\quad(\text{by Lemma \ref{lemmarho0}(ii)}) \\
					&> r^{i_0} - \rho_F(a^i - a^{i_0}) \quad(\text{by \eqref{notin} and \eqref{geneball}}) \\
					&= (\|F\|\|F^\circ\| + 1) \max_{p \in I} \rho_F(a^{i_0} - a^p) - \rho_F(a^i - a^{i_0}) \quad(\text{by \eqref{indexradius}}) \\
					&\ge (\|F\|\|F^\circ\| + 1) \max_{p \in I} \rho_F(a^{i_0} - a^p) - \|F\|\|F^\circ\| \rho_F(a^{i_0} - a^i)\\ & \qquad(\text{by Lemma \ref{lemmarho0}(viii)})  \\
					& \ge\, \max_{p \in I} \rho_F(a^{i_0} - a^p),
				\end{aligned}
		\end{equation*}}
which yields
		\begin{equation}\label{mauthuan1_final}
			\rho_F(y^k - a^i) > \rho_F(a^{i_0} - a^i)\quad \text{for all } i\in I.
		\end{equation}
		 Consider the attraction set of $y^k$ 	(see Definition~\ref{attraction_1})
		\[
		A[y^k] = \left\{ i \in I \mid \rho_F(y^k-a^i)=\min_{j\in J}\rho_F(y^j-a^i)\right\}
		\]
	and distinguish the following two cases.
		
		\textbf{Case 1: $A[y^k] \neq \emptyset$.} Then there exists an index $i^* \in I$ such that 
		\begin{equation}\label{eq22_a}
		\rho_F(y^k - a^{i^*}) = \min_{j \in J} \rho_F(y^j - a^{i^*}).	
		\end{equation}
		Set
		$$\bar{x} =(\bar x^1,\dots,\bar x^{k-1},\bar x^k)= (y^1, \dots, y^{k-1}, a^{i_0}).$$
		This together with \eqref{mauthuan1_final} and \eqref{eq22_a} implies that
		\begin{equation*}\label{25'}
			\begin{array}{ll}
				\min\limits_{j \in J} \rho_F(y^j - a^{i^*}) = \rho_F(y^k - a^{i^*}) & > \rho_F(a^{i_0} - a^{i^*})\\
				&= \rho_F(\bar{x}^k - a^{i^*})\\
				& \ge \min\limits_{j \in J} \rho_F(\bar{x}^j - a^{i^*}).    
			\end{array}
		\end{equation*}
		Hence, using~\eqref{mauthuan1_final} one again, we can evaluate the value of the objective function of~\eqref{maxminn} at the pair $(y,a)$ as follows
		\begin{align*}
			f(y, a) &= \sum_{i \in I \setminus \{i^*\}} \min_{j \in J} \rho_F(y^j - a^i) + \min_{j \in J} \rho_F(y^j - a^{i^*}) \\
			& > \sum_{i \in I \setminus \{i^*\}} \min_{j \in J} \rho_F(y^j - a^i) + \min_{j \in J} \rho_F(\bar{x}^j - a^{i^*}) \\
			&\ge \sum_{i \in I \setminus \{i^*\}} \min \left\{ \min_{j \in J\setminus\{k\}} \rho_F(y^j - a^i),\, \rho_F(a^{i_0} - a^i) \right\} + \min_{j \in J} \rho_F(\bar{x}^j - a^{i^*}) \\
			&= \sum_{i \in I \setminus \{i^*\}} \min \left\{ \min_{j \in J\setminus\{k\}} \rho_F(\bar x^j - a^i),\, \rho_F({\bar x}^k - a^i) \right\} + \min_{j \in J} \rho_F(\bar{x}^j - a^{i^*}) \\
			&= \sum_{i\in I}\min_{j \in J} \rho_F(\bar{x}^j - a^i)\\
			& = f(\bar{x},a).
		\end{align*}
		Thus, we have $f(y, a) > f(\bar{x}, a)$. This inequality contradicts the fact that $y$ is a global optimal solution to problem \eqref{maxminn}.
		
		\textbf{Case 2: $A[y^k] = \emptyset$.} Then we have $$\rho_F(y^k - a^i) > \min_{j \in J} \rho_F(y^j - a^i)\;\text{ for all }i \in I.$$ Hence, the objective value is unchanged if we remove the facility $y^k$, that is,
		\begin{equation}\label{eq21}
			f(y, a) = \sum_{i\in I} \min_{j \in J\setminus\{k\}} \rho_F(y^j - a^i). 
		\end{equation}
		As the points $a^1,\dots,a^m$ are pairwise distinct and $k-1<m$, while the set $\{y^1,\dots,y^{k-1}\}$ contains at most $k-1$ points, there exists $i^*\in I$ such that
		\[
		a^{i^*}\notin \{y^1,\dots,y^{k-1}\}.
		\]
		Therefore,
		\begin{equation}\label{eq22}
		\min_{j\in J\setminus\{k\}}\rho_F(y^j-a^{i^*})>0.
		\end{equation}
			Then by setting $\hat{x}=(\hat x^1,\dots, \hat{x}^{k-1},\hat x^k) = (y^1, \dots, y^{k-1}, a^{i^*})$ (i.e., replacing the non-attractive facility $y^k$ by the data point $a^{i^*}$) we obtain
		\begin{align*}
			f(y, a) &= \sum_{i \in I \setminus \{i^*\}} \min_{j \in J\setminus\{k\}} \rho_F(y^j - a^i) + \min_{j \in J\setminus\{k\}} \rho_F(y^j - a^{i^*})\quad(\text{by \eqref{eq21}}) \\
			&> \sum_{i \in I \setminus \{i^*\}} \min_{j \in J\setminus\{k\}} \rho_F(y^j - a^i) + \rho_F(a^{i^*} - a^{i^*}) \quad(\text{by \eqref{eq22}})\\
			&\ge \sum_{i \in I \setminus \{i^*\}} \min \left\{ \min_{j \in J\setminus\{k\}} \rho_F(y^j - a^i), \rho_F(a^{i^*} - a^i) \right\} + \rho_F(a^{i^*} - a^{i^*}) \\
			&= \sum_{i\in I} \min_{j \in J} \rho_F(\hat{x}^j - a^i)\\
			& = f(\hat{x}, a).
		\end{align*}
		This contradicts the fact that $y \in S(a)$.
		
		 We have thus proved that~\eqref{coverball} holds.

		It remains to prove that the solution set~$S(a)$ is compact. Since $\bigcap\limits_{i\in I}\mathbb B_F[a^i;r^i]$ is compact by Lemma \ref{lm2}, it follows from Remark~\ref{remarkcompact} that the set
		$		\left(\bigcap\limits_{i\in I}\mathbb B_F[a^i;r^i]\right)^k$
		is compact. By \eqref{coverball}, we have
		\[
		S(a)\subset \left(\bigcap_{i\in I}\mathbb B_F[a^i;r^i]\right)^k.
		\]
		Moreover, $S(a)$ is closed by Theorem 3.2 in \cite{LNTV}. Therefore, $S(a)$ is a closed subset of a compact set, and hence it is compact. 		
	\end{proof}

\begin{Definition}\label{defsolutionfunctions} The set-valued mapping $S: \mathbb{R}^{nm} \rightrightarrows \mathbb{R}^{nk}$, which assigns to each vector $a=(a^1,\dots,a^m)\in\mathbb{R}^{nm}$ the set
	\begin{equation}\label{solutionfunctionwp}
		S(a)=\left\{\bar{x}\in\mathbb{R}^{nk}\mid f(\bar{x},a)=\mathcal{V}(a)\right\},
	\end{equation} is called the \textit{global solution mapping} of the parametric generalized multi-source Weber problem problem~\eqref{maxminn}.
\end{Definition}	

The next example shows that the global solution mapping $S$ may not be lower semicontinuous at some point $\bar{a}=(\bar a^1,\dots,\bar a^m)\in \mathbb{R}^{nm}$, even when $\bar a^i \neq \bar a^p$ for all $i, p \in I$ with $i \neq p$.

\begin{example}\label{ex_LSC_failure}
	Consider problem \eqref{maxminn} with $n=1$, $m=3$, $k=2$, $F=\mathbb{B}[0;1]$, and $\bar a=(0,1,2)$. A direct computation shows that $\mathcal{V}(\bar a)=1$ and 
	\[
	S(\bar a)= \big( [0,1] \times \{2\} \big) \cup \big( \{0\} \times [1,2] \big).
	\]
	For each $\ell\in\mathbb{N}$, let $a_\ell = (0, 1, 2 + \frac{1}{\ell})$. It is clear that $a_\ell \to \bar a$ as $\ell \to \infty$. Since the distance between the first two points is $1$ and the distance between the last two points is $1 + \frac{1}{\ell}$, we have $\mathcal{V}(a_\ell) = 1$ and 
	\[
	S(a_\ell) = [0,1] \times \left\{ 2 + \frac{1}{\ell} \right\}
	\]
	for all $\ell \in \mathbb{N}$. Now, consider the open set 
	\[
	V = \left( -\frac{1}{2}, \frac{1}{2} \right) \times \left( 1, \frac{3}{2} \right) \subset \mathbb{R}^2.
	\]
	Since $(0, \frac{5}{4}) \in S(\bar a) \cap V$, we see that $S(\bar a) \cap V \neq \emptyset$. However, for any $(x_1, x_2) \in S(a_\ell)$, it holds that $x_2 = 2 + \frac{1}{\ell} > 2$. Thus, $S(a_\ell) \cap V = \emptyset$ for all $\ell \in \mathbb{N}$, which means that $S^k$ is not lower semicontinuous at $\bar a$.
\end{example}

\begin{remark} \label{reinner}
		Example~4.4 also shows that the global solution mapping \(S\) is not
		inner semicontinuous at the graph point $\big(\bar a,(0,5/4)\big)\in \operatorname{gph}S.$
		Indeed, the open set
		$$
		V=\left(-\frac12,\frac12\right)\times\left(1,\frac32\right)
		$$
		is a neighborhood of \((0,5/4)\), while
		\[
		S(a_\ell)\cap V=\emptyset \quad \text{for all } \ell\in\mathbb N,
		\]
		where \(a_\ell=(0,1,2+1/\ell)\to \bar a\). Hence \(S\) fails to be inner
		semicontinuous at \(\big(\bar a,(0,5/4)\big)\). Since lower semicontinuity
		of a set-valued mapping at a parameter is equivalent to inner semicontinuity
		at every graph point composed from the parameter and an element of the corresponding image set, this provides another proof
		that~\(S\) is not lower semicontinuous at~\(\bar a\).
\end{remark}

	Despite to the lack of lower semicontinuity,   the global solution mapping of the GMWP is upper semicontinuous under a mild assumption.
	
	\begin{theorem}\label{uscw}
		The global solution mapping $S$ of problem \eqref{maxminn} is upper semicontinuous at every point $\bar{a}=(\bar{a}^1,\ldots,\bar{a}^m)\in \mathbb{R}^{nm}$ satisfying the condition $\bar{a}^i \neq \bar{a}^p$ for all $i, p \in I$ with $i \neq p$.
	\end{theorem}
	
	\begin{proof}
		Let $\bar{a}=(\bar{a}^1,\ldots,\bar{a}^m)\in \mathbb{R}^{nm}$ be such that $\bar{a}^i \neq \bar{a}^p$ for all $i \neq p$ in $I$. Then one can choose $r>0$ sufficiently small so that, for every
		$$a=(a^1,\ldots,a^m)\in \prod\limits_{p\in I}\mathbb B[\bar a^p;r],$$
		the perturbed demand points remain pairwise distinct,
 i.e., $a^i \neq a^p$ whenever $i \neq p$.
		Define the diameter of the union of these neighborhoods as
		\[
		d = \text{diam}\left( \bigcup_{p \in I} \mathbb{B}[\bar{a}^p; r] \right) > 0.
		\]
		Fix any $a=(a^1,\ldots,a^m) \in \prod\limits_{i \in I} \mathbb{B}[\bar{a}^i; r]$.
		Since $\{a^1, \dots, a^m\}\subset \bigcup\limits_{i \in I} \mathbb{B}[\bar{a}^i; r]$, we have 
		\begin{equation*}\label{diam}
					\|a^i -  a^p\| \le d\;\text{ for all }i, p \in I.
		\end{equation*}
		Combining this with Lemma \ref{lemmarho0}(vii), we obtain 
		\begin{equation}\label{key311}
			\rho_F(a^i -  a^p) \le \|F^\circ\| \|a^i -  a^p\| \le \|F^\circ\| d \quad \text{for all }i, p \in I.
		\end{equation}
	Furthermore, it follows from Theorem~\ref{Compactball} (see \eqref{coverball}) and a fact obtained in the proof of Lemma~\ref{lm2}) that, for any $x = (x^1, \dots, x^k) \in S(a)$, 
		\begin{equation}\label{key313}
			x^j \in \bigcap_{i\in I}\mathbb{B}_F[a^{i}; r^{i}] \subset \bigcap_{i\in I}\mathbb{B}\left[ a^{i}; \|F\|r^{i} \right] \quad \text{for all }j \in J
		\end{equation} with $r^i$ being defined by~\eqref{indexradius} for $i\in I$. Then for any $i\in I$ and $j\in J$, we have
		\[
		\begin{aligned}
			\|x^j - \bar{a}^{i}\| &\le \|x^j - a^{i}\| + \|a^{i} - \bar{a}^{i}\| \\
			&\le \|F\|r^i+r\quad(\text{by \eqref{key313} and } a^i\in \mathbb{B}[\bar a^i;r])\\
			&= \|F\|(\|F\|\|F^\circ\| + 1) \max_{p \in I} \rho_F(a^{i} - a^p) + r \quad (\text{by \eqref{indexradius}})\\
			&\le \|F\|\|F^\circ\|(\|F\|\|F^\circ\| + 1) d + r \quad (\text{by \eqref{key311}}).
		\end{aligned}
		\]
		It follows that $$x^j \in \bigcap_{i\in I}\mathbb{B}[\bar{a}^{i}; \gamma]\;\text{ for all }j \in J,$$
		where $\gamma = \|F\|\|F^\circ\|(\|F\|\|F^\circ\| + 1) d + r$. 
		Thus, 
		\begin{equation}\label{key314}
			S(a) \subset \left(\bigcap_{i\in I}\mathbb{B}[\bar{a}^{i}; \gamma]\right)^k\; \text{for all }a \in \prod_{i \in I} \mathbb{B}[\bar{a}^i; r].
		\end{equation}
		
		Note that $\bigl(\bigcap\limits_{i\in I}\mathbb{B}[\bar{a}^{i}, \gamma]\bigr)^k$ is a compact subset of $\mathbb{R}^{nk}$, and both the objective function~$f$ and the optimal value function~$\mathcal{V}$ of~\eqref{maxminn} are continuous (see Proposition~\ref{proobjectivefuntion} and Theorem~\ref{thm:vkW_lipschitz}). Hence, applying Lemma~\ref{upper} to~\eqref{maxminn} and using~\eqref{key314}, we can infer that the global solution mapping~$S$ is upper semicontinuous at~$\bar{a}$.
	\end{proof}
	
	\begin{remark}
		When $\rho_F$ is the Euclidean norm, Theorem \ref{uscw} reduces to Theorem 4.1 in    \cite{Cuong2024}.
	\end{remark}
	
	\section{Stability analysis of local solution mappings}\label{sec5}
	
In contrast to the global solution mapping, which  can neither lower semicontinuous nor inner lower semicontinuous ( see Example~\ref{ex_LSC_failure} and Remark~\ref{reinner}), the local solution mapping may possess these property under certain regularity conditions. 
	
	\begin{Definition} \label{deflocalsolutionmapping}
			The \textit{local solution mapping} $S_{loc}: \mathbb{R}^{nm} \rightrightarrows \mathbb{R}^{nk}$ of (GMWP) is defined by
		\begin{equation} \label{localsolutionfunction}
			S_{loc}(a) = \left\{ \bar{x} \in \mathbb{R}^{nk} \mid \bar{x} \text{ is a local optimal solution of \eqref{maxminn}}\right\}.
		\end{equation}
	\end{Definition}

 We now prove that inner semicontinuity of $S_{loc}(\cdot)$ is available at the graph points generated by the strict local optimal solutions of~(GMWP).
	
	\begin{theorem} \label{thm:LSC_local}
		Let $\bar{x} \in S_{loc}(\bar{a})$  be a strict local optimal solution at $\bar{a}$ (see Definition \ref{deflocalsolution}).  Then the local solution mapping $S_{loc}$ is inner semicontinuous at $(\bar{a},\bar{x})$.
	\end{theorem}
	\begin{proof}
		 Suppose to the contrary that $S_{loc}$ is not inner semicontinuous at $(\bar{a}, \bar{x})$. According to Definition \ref{defusclsc}(iv), this means that there exists an open neighborhood $V$ of $\bar{x}$ such that for every neighborhood $U$ of $\bar{a}$, there exists at least one $a \in U$ satisfying $S_{loc}(a) \cap V = \emptyset.$ Then, by considering the neighborhoods $U_\ell = \mathbb{B}(\bar{a}; 1/\ell)$ with $\ell \in \mathbb{N}$, we can construct a sequence $(a_\ell)$ such that $a_\ell \to \bar{a}$ and 
		\begin{equation} \label{pf:isc_contradiction}
			S_{loc}(a_\ell) \cap V = \emptyset \quad \text{for all }\, \ell \in \mathbb{N}.
		\end{equation}
		Since $\bar{x}$ is a strict local optimal solution at $\bar{a}$, there exists $\varepsilon > 0$ such that 
			\begin{equation}\label{s_localsolution}
				f(\bar{x}, a)<f(x, a)  \quad \text{for all }\,  x \in \mathbb{B}(\bar{x}; \varepsilon) \setminus \{\bar{x}\}.      
			\end{equation}
	Without loss of generality, we choose $\varepsilon$ small enough so that $\mathbb{B}(\bar{x}; \varepsilon/2) \subset V$. For each $\ell$, consider the restricted optimization problem
		\[ \min_{x \in \mathbb{B}[\bar{x}; \varepsilon/2]} f(x, a_\ell). \]
		Since $f(\cdot, a_\ell)$ is continuous (by Proposition \ref{proobjectivefuntion}) and the set $\mathbb{B}[\bar{x}; \varepsilon/2]$ is compact, it follows from the Weierstrass Theorem that there is $x_\ell\in\mathbb{B}[\bar{x}; \varepsilon/2]$ such that
		\begin{equation}\label{localsolution1}
			\min_{x \in \mathbb{B}[\bar{x}; \varepsilon/2]} f(x, a_\ell)=f(x_\ell,a_\ell).
		\end{equation}
		Again, by the compactness of $\mathbb{B}[\bar{x}; \varepsilon/2]$, there exists a subsequence of $(x_\ell)$ (still denoted by $(x_\ell)$) converging to some $\hat{x} \in \mathbb{B}[\bar{x}; \varepsilon/2]$. Since $f$ is continuous on $\mathbb{R}^{nk}\times \mathbb{R}^{nm}$ by Proposition~\ref{proobjectivefuntion}, we have $f(x_\ell, a_\ell) \to f(\hat{x}, \bar{a}).$
		Because $x_\ell$ is an optimal solution of $f(\cdot, a_\ell)$ on $\mathbb{B}[\bar{x}; \varepsilon/2]$, it holds that $$f(x_\ell, a_\ell) \leq f(x, a_\ell)\;\text{ for all }x \in \mathbb{B}[\bar{x}; \varepsilon/2].$$ 
		Taking the limit as $\ell \to \infty$ yields
		\begin{equation}\label{hat_x} f(\hat{x}, \bar{a}) \leq f(x, \bar{a}) \; \text{ for all }\, x \in \mathbb{B}[\bar{x}; \varepsilon/2].\end{equation}
		 If $\hat{x}\neq \bar{x}$, then by~\eqref{s_localsolution} we see that $f(\bar{x}, a)<f(\hat x, a)$. This contradicts~\eqref{hat_x}. Hence, we must have $\hat{x} = \bar{x}$. 
		It follows that $x_\ell \to \bar{x}$, and so we get  
		$$x_\ell \in \mathbb{B}(\bar{x}; \varepsilon/2)\subset V\;\text{ for all }\ell\text{ sufficiently large}.$$ 
		Combining this with \eqref{localsolution1}, we conclude that $x_\ell$ is a local optimal solution of $f(\cdot,a_\ell)$ on $\mathbb{B}(\bar{x}; \varepsilon/2)$ for all $\ell$ sufficiently large, i.e., $x_\ell\in S_{loc}(a_\ell).$ Therefore, 
		$$S_{loc}(a_\ell) \cap V \neq \emptyset,$$ 
		which directly contradicts our assumption in \eqref{pf:isc_contradiction}. The proof is complete.
	\end{proof}

	 Strict local optimality in Theorem~\ref{thm:LSC_local} provides a convenient sufficient condition for the inner semicontinuity of the local solution mapping. However, this condition is not necessary in general, as illustrated by the following example.
	
	\begin{example}\label{ex:ISC_non_strict}
		Consider (GMWP) in $\mathbb{R}$ with $k=1$ and $m=2$. Let the data points be $\bar a^1=-1$ and $\bar a^2=1$, and let the distance be the absolute value $\rho_F=|\cdot|$. The corresponding objective function is
		\[
		f(x,\bar a)=|x+1|+|x-1|.
		\]
		It is straightforward to verify that the set of global optimal solutions is the interval $S(\bar a)=[-1,1]$. Hence, every point in this segment is a (non-strict) local optimal solution.
		
		Fix $\bar x=1 \in S_{loc}(\bar a)$ and let $V$ be an arbitrary neighborhood of $\bar x$. Then, there exists $\varepsilon > 0$ such that $(1-\varepsilon, 1+\varepsilon) \subset V$. We define a neighborhood $U$ of $\bar{a} = (-1, 1)$ in $\mathbb{R}^2$ as
		\[
		U = (-1-\varepsilon/2, -1+\varepsilon/2) \times (1-\varepsilon/2, 1+\varepsilon/2).
		\]
		For any perturbed data $a = (a^1, a^2) \in U$, the objective function becomes 
		$$f(x, a) = |x - a^1| + |x - a^2|.$$ 
		A direct computation shows that
		\[
		S(a)=[a^1,a^2],
		\]
		and in particular,
		\[
		a^2\in S_{loc}(a)\cap V.
		\]
		Therefore, $S_{loc}(a)\cap V\neq\emptyset$ for all $a\in U$,
		which proves that the local solution mapping $S_{loc}$ is inner semicontinuous at $(\bar a, \bar x)$, despite the fact that $\bar x$ is not a strict local optimal solution.
	\end{example}

	Next, we recall some classical notions regarding strict convexity; for further details, see \cite{Bauschke,boyd,mordukhovich2023easy}.
	
	\begin{Definition}\label{strictlyconvexfuntion}
		Let $\Omega \subseteq \mathbb{R}^n$ be a convex set. A function $\phi: \Omega \to \mathbb{R}$ is said to be \textit{strictly convex} if for any two distinct points $x, y \in \Omega$ and any $\lambda \in (0, 1)$, the following strict inequality holds:$$\phi(\lambda x + (1 - \lambda) y) < \lambda \phi(x) + (1 - \lambda) \phi(y).$$
	\end{Definition}
	
	\begin{Definition}
		A set $\Omega\subset \mathbb{R}^n$ is said to be \textit{strictly convex} if for any two distinct points $x, y \in\Omega$  and any $\lambda \in (0, 1)$, one has 
		$$\lambda x + (1 - \lambda)y \in \operatorname{int} (\Omega).$$
	\end{Definition}
	
	The following lemma is fundamental for establishing another result on the inner semicontinuity of the local solution mapping for the GMWP.
	
	\begin{lemma}\label{lemmaNam}
		Let $\Omega$ and $\{a^1,\dots,a^q\}$ be subsets of $\mathbb{R}^n$ such that $\Omega$ is convex, and let $\phi:\Omega\to \mathbb{R}$ be defined by
		\begin{equation}\label{phi_z}\phi(z)=\sum_{i=1}^q\rho_F(z-a^i)\;\text{ for } z\in \Omega,
			\end{equation}
		where $\rho_F$ is the Minkowski function defined in\eqref{Minkowski_f}.  
		Assume that $F$ is strictly convex and the points $\{a^1,\dots,a^q\}$ are not collinear.
		Then the function $\phi$ is strictly convex.  
	\end{lemma}
	\begin{proof}
		Since the Minkowski function $\rho_F$ is convex (see, e.g., \cite[Proposition 6.10]{mordukhovich2023easy}) and the sum of convex functions remains convex, the function $\phi(\cdot)$ in~\eqref{phi_z} is convex on $\Omega$. To establish the strict convexity of~$\phi$, we proceed by contradiction. Suppose there exist two distinct points $\bar{x}, \bar{y} \in \Omega$ and some $\lambda \in (0, 1)$ such that\begin{equation}\label{eq:equality_phi}\phi(\lambda \bar{x} + (1 - \lambda) \bar{y}) = \lambda \phi(\bar{x}) + (1 - \lambda) \phi(\bar{y}).\end{equation}Due to the convexity of each individual term $\rho_F(z - a^i)$, the equality \eqref{eq:equality_phi} holds if and only if\begin{equation}\label{eq:equality_rho}\rho_F(\lambda(\bar{x} - a^i) + (1 - \lambda)(\bar{y} - a^i)) = \lambda \rho_F(\bar{x} - a^i) + (1 - \lambda) \rho_F(\bar{y} - a^i)\end{equation}for all $i \in \{1, \dots, q\}$. Using the homogeneous property of $\rho_F(\cdot)$ (see Lemma~\eqref{lemmarho0}), we can rewrite the expression on the right-hand side of~\eqref{eq:equality_rho} as $$\rho_F\left(\lambda(\bar{x} - a^i)\right) + \rho_F\left((1 - \lambda) (\bar{y} - a^i)\right).$$ Therefore, if $\bar{x} \neq a^i$ and $\bar{y} \neq a^i$, then by~\eqref{eq:equality_rho}, the strict convexity of $F$, and~\cite[Proposition 8.13]{mordukhovich2023easy} we can find $\alpha_i > 0$ such that $$\alpha_i \lambda (\bar{x} - a^i) = (1 - \lambda) (\bar{y} - a^i).$$ Rewrite the last equality as $$\bar{x} - a^i = \gamma^i (\bar{y} - a^i), \quad \text{where } \gamma^i = \frac{1 - \lambda}{\alpha_i \lambda}.$$
		Since $\bar{x} \neq \bar{y}$, it follows that $\gamma^i \neq 1$. Solving for $a^i$ yields 
		$$a^i = \frac{1}{1 - \gamma^i} \bar{x} +\left(1- \frac{1}{1 - \gamma^i}\right) \bar{y}.$$
		This shows that $a^i$ lies on the line $\mathcal{L}(\bar{x}, \bar{y}) = \{t\bar{x} + (1-t)\bar{y} \mid t \in \mathbb{R} \}$. In the cases where $\bar{x} = a^i$ or $\bar{y} = a^i$, it is trivial that $a^i \in \mathcal{L}(\bar{x}, \bar{y})$. Consequently, all  the points $\{a^1, \dots, a^q\}$ must lie on the same line $\mathcal{L}(\bar{x}, \bar{y})$, which contradicts the assumption. Thus, $\phi$ must be strictly convex.
	\end{proof}

 The next sufficient condition for the inner semicontinuity of the local solution mapping for the GMWP is the second main result of this section.
	
	\begin{theorem}\label{main_result_3}
		Consider the GMWP defined in~\eqref{maxminn}, where the set $F$ associated with the Minkowski function $\rho_F$ is strictly convex. Let $(\bar a, \bar x) \in \operatorname{gph} (S_{loc})$, where $\bar{a}=(\bar{a}^1, \dots, \bar{a}^m) \in \mathbb{R}^{nm}$ and $\bar{x}=(\bar{x}^1, \dots, \bar{x}^k) \in \mathbb{R}^{nk}$. Let 
		$$\bar A=\{{\bar a}^1,\ldots,{\bar a}^m \}\quad \text{and}\quad\widehat J=\{j\in J\mid \bar A[\bar x^{j}]\neq\emptyset\},$$ 
		where $\bar A[\bar x^{j}]$ (defined in \eqref{attraction}) is the attraction set of $\bar x^{j}$ with respect to $\bar A$.
	  If the conditions
		\begin{enumerate}
			\item[{\rm (i)}] $\bar A[\bar{x}^{j_1}] \cap \bar A[\bar{x}^{j_2}]=\emptyset$ whenever $j_1, j_2 \in J$ and $j_1 \neq j_2$;
			\item[{\rm (ii)}] For each $j \in \widehat{J}$, the data points  from $\bar{A}[\bar{x}^j]$ are not collinear
		\end{enumerate} are satisfied, then the local solution mapping $S_{loc}$ of \eqref{maxminn} is inner semicontinuous at $(\bar a, \bar x)$.
	\end{theorem}

	\begin{proof}
		Since $\bar x$ is a local optimal solution of problem~\eqref{maxminn} at $\bar a$, there exists $\varepsilon_1>0$ such that
		$$f(\bar x,\bar a)\le f( x,\bar a)\;\text{ for all } x\in \mathbb{B}(\bar{x};\varepsilon_1).$$
		 Let us show that $\bar x$ is a strict local optimal solution of the GMWP at $\bar a$.
		
		By (i) and the definition of the attraction set $\bar A[\bar x^j]$, for every $j\in\widehat{J}$ we have
		\begin{equation}\label{ineq:strict_attraction}
			\rho_F(\bar x^j-\bar a^i)
			<\min_{l\in J\setminus\{j\}}\rho_F(\bar x^l-\bar a^i)\;\text{ for all } \bar a^i\in\bar A[\bar x^j],
		\end{equation} where the convention $\min\emptyset=+\infty$ is used. Meanwhile, for every $j\in J\setminus \widehat{J}$,  we have
		\begin{equation}\label{ineq:strict_attraction1}
			\rho_F(\bar x^j-\bar a^i)
			>\min_{l\in J\setminus\{j\}}\rho_F(\bar x^l-\bar a^i)\;\text{ for all }\bar a^i\in \bar A.
		\end{equation}
		By the continuity of $\rho_F$ (see Lemma~\ref{lemmarho0}(iv)), the strict inequalities in \eqref{ineq:strict_attraction} and \eqref{ineq:strict_attraction1} remain valid in a sufficiently small neighborhood of each $\bar x^j$.
		Therefore, there exists $\varepsilon_2>0$ such that when $\bar{x}=(\bar{x}^1, \dots, \bar{x}^k)$ moves to any new facilities system
		$x=(x^1,\dots,x^k)\in \mathbb B(\bar x;\varepsilon_2)$, the attraction sets remain unchanged, namely,
		\[
		A[x^j]=\bar A[\bar x^j]\quad \text{for all } j\in J.
		\]
 Hence, the objective function admits the local representation
		$$f(x,\bar a)=\sum_{j \in \widehat J} \phi_j(x^j)\;\text{ for all } x\in \mathbb{B}(\bar{x};\varepsilon_2), 
		$$
		where $$\phi_j(x^j) = \sum_{\bar a^i \in A[\bar x^j]} \rho_F(x^j - \bar a^i).$$
		Under assumption (ii), it follows from  Lemma~\ref{lemmaNam} that for every $j\in \widehat J$, $\phi_j$ is strictly convex on
		$\mathbb B(\bar x^j;\varepsilon_2)$.
		Since $f(\cdot,\bar a)$ is a sum of strictly convex functions
		$\phi_j$ in mutually independent variables $x^j$, 
		$f(\cdot,\bar a)$ is also strictly convex on $\mathbb{B}(\bar{x};\varepsilon)$ with $\varepsilon=\min\{\varepsilon_1,\varepsilon_2\}$.
		Note that a local optimal solution of a strictly convex function is necessarily a strict local optimal solution, see, e.g., \cite[p. 137]{boyd}. Therefore, $\bar{x}$ is a strict local optimal solution to the GMWP at $\bar{a}$. 
		
		 Now, it is clear that all the assumptions of Theorem~\ref{thm:LSC_local} are  satisfied. By this theorem, we conclude that the local solution mapping $S_{loc}$ is inner semicontinuous at $(\bar a, \bar x)$. The proof is complete.
	\end{proof}

	\begin{remark}
		It is worth emphasizing that Theorem~\ref{main_result_3} extends
		Theorem~5.1 in~\cite{Cuong2024} by replacing the Euclidean norm with a more
		general Minkowski function framework. Furthermore, our approach employs a different proof technique and indirectly provides answers to two open questions raised in \cite{Cuong2024}. Specifically, we are able to remove the assumption that $\bar{a}^i \neq \bar{a}^p$ for $i \neq p$, as well as condition (ii) of Theorem 5.1 in \cite{Cuong2024}.
	\end{remark}

	The final example of this paper addresses Question 2 raised in \cite{Cuong2024}. Specifically, it shows that condition (ii) of Theorem 5.1 in \cite{Cuong2024} is not a necessary requirement for the inner semicontinuity of the local solution mapping.
	
	\begin{example}
		Let $k=2$ and $$\bar A = \{ \bar a^1=(-1,0), \bar a^2=(1,0), \bar a^3=(0,\sqrt{3}), \bar a^4=(100,0) \}.$$ 
		Let $\rho_F$ be the Euclidean norm. Consider the multi-source Weber problem:
		\begin{equation}\label{phanvd}
			f(\bar a) = \min_{x^1, x^2 \in \mathbb{R}^2} \sum_{i=1}^4 \min_{j \in \{1,2\}} \|x^j - \bar a^i\|.    
		\end{equation}
		A direct computation shows that for $\epsilon=1/100$, the point $$\bar{x}=(\bar x^1,\bar x^2) = \left( \left(0, \frac{\sqrt{3}}{3}\right), (100, 0) \right)$$ is a local optimal solution of \eqref{phanvd}. Moreover, It is straightforward to verify that conditions (i) and (ii) of Theorem \ref{main_result_3} are satisfied. By this theorem, the local solution mapping $S_{loc}$ of \eqref{phanvd} is inner semicontinuous at $(\bar a,\bar x).$ However, observe that $\bar x^2\in \bar A$. This confirms that condition (ii) of Theorem 5.1 in \cite{Cuong2024}, which requires $\bar x^j \notin \bar{A}$ for all $j=1,2$, is not essential for the result to hold.
	\end{example}
	
	\section{Conclusions}\label{sec6}
	
  Using the Minkowski function, we have provided a detailed stability analysis of the generalized  multi-source Weber problems. For future research, it would be of interest to utilize the explicit Lipschitz constants derived in this work to develop robust numerical algorithms. Additionally, exploring the stability of these problems in infinite-dimensional spaces or under more complex constraints remains a promising direction.
	
\end{document}